\pdfoutput=1
\documentclass[11pt]{amsart}

\usepackage{graphicx}
\usepackage[english]{babel}
\usepackage[T1]{fontenc}
\usepackage[latin1]{inputenc}
\usepackage{amsfonts}
\usepackage{amssymb}
\usepackage{amsthm}
\usepackage{amsmath}
\usepackage{tikz}
\usepackage{float}
\usepackage{enumerate}
\usepackage{accents}
\usepackage{mathtools}
\usepackage{mathrsfs}
\usepackage{comment}
\usepackage{afterpage}
\usepackage{bbm}
\usepackage{dsfont}
\usepackage{stmaryrd}
\usepackage{hyperref}
\usepackage{mleftright}
\usepackage{subfig}
\usepackage{soul}
\usepackage{tikz-cd} 
\usepackage{fullpage}
\usepackage{cleveref}

\numberwithin{equation}{section}

\theoremstyle{plain}
\newtheorem{thm}{Theorem}[section]
\newtheorem*{thm*}{Theorem}
\newtheorem{prop}[thm]{Proposition}
\newtheorem*{prop*}{Proposition}
\newtheorem{cor}[thm]{Corollary}
\newtheorem*{cor*}{Corollary}
\newtheorem{lem}[thm]{Lemma}

\newtheorem{thmintro}{Theorem}

\newtheorem{corintro}[thmintro]{Corollary}

\theoremstyle{definition}
\newtheorem{defn}[thm]{Definition}
\newtheorem*{defn*}{Definition}

\newtheorem{rmk}[thm]{Remark}
\newtheorem*{rmk*}{Remarks}

\newtheorem*{conj*}{Conjecture}
\newtheorem*{quest*}{Question}

\newtheoremstyle{blue-environment}{}{}{}{}{\color{blue}\bfseries}{.}{ }{}
\theoremstyle{blue-environment}

\newcommand{\acts}{\curvearrowright}
\newcommand{\ra}{\rightarrow}
\newcommand{\Ra}{\Rightarrow}

\newcommand{\sq}{\subseteq}

\newcommand{\x}{\times}
\renewcommand{\o}{\circ}

\newcommand{\mc}{\mathcal}
\newcommand{\mf}{\mathfrak}
\newcommand{\mscr}{\mathscr}

\newcommand{\R}{\mathbb{R}}
\newcommand{\Z}{\mathbb{Z}}
\newcommand{\N}{\mathbb{N}}

\newcommand{\Om}{\Omega}

\newcommand{\g}{\gamma}
\newcommand{\G}{\Gamma}

\newcommand{\X}{\mc{X}}

\DeclareMathOperator{\Fix}{Fix}
\DeclareMathOperator{\Min}{Min}

\begin{document}

\title{$\R$--trees and accessibility over arc-stabilisers} 

\author[E.\,Fioravanti]{Elia Fioravanti}\address{Institute of Algebra and Geometry, Karlsruhe Institute of Technology}\email{elia.fioravanti@kit.edu} 
\thanks{The author is supported by Emmy Noether grant 515507199 and grant 541703614 of the DFG}

\begin{abstract}
    Let $G\acts T$ be a minimal action on an $\R$--tree with $G$ finitely presented. Assuming that $G$ is accessible over the family of arc-stabilisers of $T$, and that arc-stabilisers are finitely generated, we give a description of the point-stabilisers of $T$ in terms of simplicial splittings of $G$. The most important consequence is that the point-stabilisers of $G\acts T$ are finitely generated. This has applications to the study of automorphisms of right-angled Artin groups and special groups.
\end{abstract}

\maketitle

\section{Introduction}

When a finitely generated group $G$ acts on a simplicial tree $T$, the properties of edge-stabilisers can usually be translated rather precisely into properties of vertex-stabilisers. Many classical instances of this phenomenon are known, regarding e.g.\ finite generation and finite presentability, type $F_n$ \cite{HW21}, quasi-convexity \cite[Proposition~1.2]{Bow-JSJ}, and undistortion \cite[Lemma~2.15]{FK22}.

A more delicate problem is to deduce information about the point-stabilisers of a $G$--action on an $\R$--tree $T$, assuming that analogous information is known for arc-stabilisers. To some extent, one can use Rips--Sela theory to approximate the $\R$--tree by $G$--actions on simplicial trees \cite{Rips-Sela,GLP1,BF-stable,Guir-CMH}, but such approximations typically only capture information about \emph{finitely generated} subgroups of point-stabilisers, and this is not the full picture a priori. When the $\R$--tree has small arc-stabilisers, very precise descriptions of point-stabilisers are known \cite{Gaboriau-Levitt,GJLL,Guirardel-complexity,Horbez17}, but this does not suffice for all applications. For instance, the properties of automorphisms of many non-hyperbolic groups are naturally encoded in actions on \emph{non-small} $\R$--trees \cite{Fio10e,Fio10a}.

When working with non-small $\R$--trees, more information about the group $G$ is needed. Roughly, it is important to know that, as we approximate the $\R$--tree by simplicial trees, the latter cannot get arbitrarily complicated. The relevant concept here is accessibility: A group $G$ is \emph{accessible} over a family of subgroups $\mc{F}$ if there is a uniform upper bound $N=N(G)$ on the number of edge-orbits of reduced simplicial $G$--trees $T$ with edge-stabilisers in the family $\mc{F}$. For instance, finitely presented groups are accessible over finite \cite{Dunwoody-acc} and even small subgroups \cite{BF-complexity}. Accessibility is also known for uniformly acylindrical splittings with arbitrary edge-stabilisers \cite{Sela-acyl-acc,Weidmann-acc}.

Our main result is the following. We explain terminology and notation right after. We also refer to \Cref{thm:acc_implies_nice} below for a more precise result.

\begin{thmintro}\label{thmintro}
    Let $G$ be finitely presented and torsion-free. Let $G\acts T$ be a minimal $\R$--tree, and let $\mc{F}$ be the family of $G$--stabilisers of arcs of $T$. Suppose that all the following are satisfied.
    \begin{enumerate}
        \item[(i)] The group $G$ is accessible over $\mc{F}_{\rm int}\cup{\rm Ess}(\mc{F},\emptyset)$.
        \item[(ii)] All elements of $\mc{F}_{\rm int}$ are finitely generated and root-closed in $G$.
        \item[(iii)] Chains of subgroups in $\mc{F}_{\rm int}\cup{\rm Ess}(\mc{F},\emptyset)$ have bounded length.
    \end{enumerate}
    Then all of the following hold.
    \begin{enumerate}
        \item The point-stabilisers of $T$ are finitely generated.
        \item There are only finitely many $G$--orbits of branch points and of directions based at branch points. In particular, there are only finitely many $G$--orbits of points $p\in T$ such that $G_p\not\in\mc{F}$.
        \item If a subgroup $H\leq G$ does not preserve a point or line in $T$, then $H$ is non-elliptic in a $(\mc{F}_{\rm int}\cup{\rm Ess}(\mc{F},\mscr{E}),\mscr{E})$--splitting of $G$, where $\mscr{E}$ is the family of subgroups of $G$ elliptic in $T$.
    \end{enumerate}
\end{thmintro}

We briefly explain here the terminology and notation used above. Given a family $\mc{F}$ of subgroups of $G$, we denote by $\mc{F}_{\rm int}$ the family of finite intersections of subgroups in $\mc{F}$. A subgroup $H\leq G$ is \emph{root-closed} if, whenever $g^n\in H$ for some $g\in G$ and $n\geq 2$, we actually have $g\in H$. By a \emph{splitting} of a group $G$, we refer to a minimal action on a simplicial tree $S$ with at least one edge and no edge-inversions. Following \cite{GL-JSJ}, we speak of an $(\mc{F},\mc{H})$--splitting, for families of subgroups $\mc{F}$ and $\mc{H}$, if all subgroups in $\mc{H}$ are elliptic in $S$, and if the $G$--stabiliser of each edge of $S$ lies in $\mc{F}$. Finally, the family ${\rm Ess}(\mc{F},\mc{H})$ consists of the subgroups of $G$ of the form $F\rtimes\langle\gamma\rangle$, where there exists a quadratically hanging vertex group $Q\leq G$ with fibre $F\in\mc{F}$ in some $(\mc{F},\mc{H})$--splitting of $G$, and where $\gamma\in Q/F$ corresponds to an essential simple closed curve on the associated surface. A \emph{branch point} in an $\R$--tree $T$ is a point $p$ such that $T\setminus\{p\}$ has at least $3$ connected components. The \emph{directions} at $p$ are the components of $T\setminus\{p\}$, or equivalently the arc germs based at $p$.

Several assumptions in the statement of \Cref{thmintro} are meant to (significantly) simplify its proof, but they are not essential for the result to hold. This is the case for finite presentability of $G$ (rather than mere finite generation), lack of torsion in $G$, and the related requirement that arc-stabilisers be root-closed.

An important application of \Cref{thmintro} is to the case when $G$ is a (compact) special group in the sense of Haglund and Wise \cite{HW08}. That is, $G$ is the fundamental group of a compact special cube complex. Examples include all right-angled Artin groups. A suitable application of the Bestvina--Paulin construction \cite{Bes88,Pau91} shows that automorphisms of special groups naturally give rise to $G$--actions on $\R$--trees whose arc-stabilisers are centralisers \cite{Fio10e}.

Denote by $\mc{VP}(G)$ the family of \emph{virtual products} within $G$, namely the subgroups that virtually split as the direct product of two infinite groups. We also denote by $\mc{Z}(G)$ the family of \emph{centralisers}, that is, the family of subgroups of the form $Z_G(A)$ for a subset $A\sq G$, where we write $Z_G(A):=\{g\in G\mid ga=ag,\ \forall a\in A\}$. We will show in \cite[Section~3]{Fio11} that all special groups are accessible over centralisers, so that this assumption can in fact be removed from the next corollary. 

\begin{corintro}\label{corintro}
    Let $G$ be a special group. Suppose that $G$ is accessible over $\mc{Z}(G)$. Let $G\acts T$ be a minimal $\R$--tree with arc-stabilisers in $\mc{Z}(G)$. Let $\mscr{E}$ be the family of elliptic subgroups of $G$.
    \begin{enumerate}
        \item For every point $p\in T$, the stabiliser $G_p$ is convex-cocompact in $G$ (and hence itself special).
        \item There are only finitely many $G$--conjugacy classes of point-stabilisers of $T$.
        \item If a subgroup $H\leq G$ is non-elliptic in $T$ while all elements of $\mc{VP}(G)$ are elliptic in $T$, then $H$ is non-elliptic in a $(\mc{Z}(G),\mscr{E})$--splitting of $G$.
    \end{enumerate}
\end{corintro}

Again, see \Cref{cor:acc_implies_nice} for a more precise result. A priori, the notion of convex-cocompactness in \Cref{corintro} depends on the choice of an embedding of $G$ into a right-angled Artin group (see \Cref{sub:cc}). However, \Cref{corintro} shows that the point-stabilisers $G_p$ are always convex-cocompact, regardless of the choice of such an embedding. Part of the reason is that centralisers are always convex-cocompact, again regardless of the choice of embedding.

\Cref{corintro} is a key ingredient in our following work \cite{Fio11}, where we describe growth rates of automorphisms of special groups $G$, including automorphisms of right-angled Artin and Coxeter groups. Roughly, the point-stabilisers of certain $\R$--trees with arc-stabilisers in $\mc{Z}(G)$ are related to the maximal subgroups of $G$ that grow at non-maximal speed under an automorphism of $G$.

We prove \Cref{thmintro} in \Cref{app:acc_R_trees}, and deduce \Cref{corintro} in \Cref{sect:special}.

\smallskip
{\bf Acknowledgements.} I am grateful to the referee for their careful reading and useful comments.

\section{Preliminaries} 

\subsection{Generalities}

Throughout the article, group actions on $\R$--trees are implicitly assumed to be by isometries, and all actions on simplicial trees are assumed to be without edge-inversions.

An action on an $\R$--tree $G\acts T$ is \emph{minimal} if there is no proper $G$--invariant subtree. If $G$ is finitely generated, then every $G$--action on an $\R$--tree without a global fixed point admits loxodromic elements, and so there exists a (unique) $G$--minimal subtree $\Min(G,T)\sq T$, namely the union of the axes of all loxodromics. The following observation is useful when dealing with subgroups of $G$ that are not finitely generated; we refer to \cite[Lemma~2.18]{Fioravanti-Kerr} for a proof.

\begin{lem}\label{lem:FK}
    Let $G\acts T$ be an $\R$--tree such that chains of arc-stabilisers have bounded length. Then, if $H\leq G$ is a subgroup, either $H$ fixes a point of $T$ or $H$ contains a loxodromic element.
\end{lem}

Let $\mc{F}$ and $\mc{H}$ be two families of subgroups of $G$. As in the Introduction, a simplicial \emph{$(\mc{F},\mc{H})$--tree} is an action on a simplicial tree $G\acts T$ such that the $G$--stabiliser of each edge of $T$ lies in $\mc{F}$, and such that every subgroup in $\mc{H}$ fixes a point of $T$. We speak of a \emph{splitting} if the action is minimal and the tree is not a single point.

\subsection{Accessibility}\label{sub:access}

Following \cite{BF-complexity}, a splitting $G\acts T$ is \emph{reduced} if there does not exist a vertex $v\in T$ whose $G$--stabiliser fixes an incident edge and acts transitively on the remaining incident edges. More weakly, we say that $G\acts T$ is \emph{irredundant} if $T$ has no degree--$2$ vertices, except for those where the vertex-stabiliser swaps the two incident edges; in other words, $T$ was not obtained from a smaller splitting of $G$ simply by subdividing an edge. Reduced splittings are irredundant, but the converse does not hold.

A group $G$ is \emph{accessible} over a family of subgroups $\mc{F}$ if there exists a number $N(G)$ such that any reduced splitting of $G$ over $\mc{F}$ has at most $N(G)$ orbits of edges. We say that $G$ is \emph{unconditionally accessible} over $\mc{F}$ if the same is true of irredundant splittings. For instance, finitely presented groups are accessible over small subgroups \cite{BF-complexity}, but already the free group $F_2$ fails to be unconditionally accessible over cyclic subgroups \cite[p.\,450]{BF-complexity}. If $G$ is (unconditionally) accessible over a family $\mc{F}$, then $G$ is also (unconditionally) accessible over all sub-families $\mc{F}'\sq\mc{F}$.

One can often promote accessibility to unconditional accessibility using the following observation. In turn, unconditional accessibility will allow us to be a bit more careless in our arguments.

\begin{lem}\label{lem:unconditional_vs_conditional_acc}
    Let $G$ be finitely generated and accessible over a family $\mc{F}$. If there is a uniform bound on the length of chains of subgroups in $\mc{F}$, then $G$ is unconditionally accessible over $\mc{F}$.
\end{lem}
\begin{proof}
    Let $G\acts T$ be an irredundant splitting over $\mc{F}$, and let $\mc{G}:=T/G$ be the quotient graph. We denote by $|\mc{G}|$ the number of edges of $\mc{G}$. Let $m$ be the largest integer such that there exists a chain $F_1\lneq\dots\lneq F_m$ in $\mc{F}$. We claim that there exists a collapse $G\acts T'$ that is reduced and has $\mc{G}':=T'/G$ with $|\mc{G}|\leq m|\mc{G}'|$. The lemma immediately follows from this.

    To prove the claim, say that a vertex $v\in\mc{G}$ is \emph{bad} if the splitting fails to be reduced at $v$: this means that $v$ has degree $2$ and its vertex group coincides with one of the two incident edge groups (exactly one, as $T$ is irredundant). There are finitely many edge paths $\pi_1,\dots,\pi_k$ in $\mc{G}$ such that the interiors of the $\pi_i$ are pairwise disjoint, all vertices in the interior of each $\pi_i$ are bad, and all bad vertices of $\mc{G}$ lie in the interior of some $\pi_i$. The endpoints of each $\pi_i$ are either vertices of degree $\geq 3$ in $\mc{G}$, or vertices of degree $\leq 2$ whose vertex group properly contains all incident edge groups. 
    
    Say that an edge of some $\pi_i$ is a \emph{top edge} if its edge group coincides with both vertex groups placed at its vertices, and a \emph{bottom edge} if its edge group is properly contained in both vertex groups; we call the remaining edges of $\pi_i$ \emph{transitional}. We need the following observations.
    \begin{enumerate}
        \item Let $\tau\sq\pi_i$ be a maximal segment all of whose edges are transitional. As we move along $\tau$, we see edge groups strictly increase (or strictly decrease, depending on the direction of movement). Thus, $\tau$ is preceded by a bottom edge of $\pi_i$ and followed by a top edge of $\pi_i$, unless some of the endpoints of $\tau$ coincide with endpoints of $\pi_i$. Moreover, $\tau$ contains at most $m$ edges (and so does the union of $\tau$ with its preceding/following bottom/top edges, when these exist).
        \item Suppose that some $\pi_i$ only contains transitional edges. Collapsing all edges of $\pi_i$ except for one (it does not matter which), we can remove all bad vertices in the interior of $\pi_i$, while the endpoints of $\pi_i$ remain distinct non-bad vertices of the new graph of groups.
        \item If we collapse a top or bottom edge of some $\pi_i$, then the resulting vertex of the new graph of groups becomes non-bad.
    \end{enumerate}
    Item~(3) shows that, up to collapsing all top and bottom edges of the $\pi_i$, we can assume that all $\pi_i$ only contain transitional edges. Item~(2) then shows that we can collapse all edges but one in each of the $\pi_i$ and thus obtain a collapse $\mc{G}'$ without bad vertices, that is, a reduced graph of groups. Finally, Item~(1) guarantees that $|\mc{G}|\leq m|\mc{G}'|$, concluding the proof.
\end{proof}

\subsection{Refinements, collapses, deformation spaces}

If $T$ and $S$ are simplicial $G$--trees, a $G$--equivariant map $\pi\colon T\ra S$ is a \emph{collapse} if it preserves alignment of triples of vertices; equivalently, $S$ is obtained from $T$ by collapsing to points the edges in some $G$--orbits. The tree $S$ is then called a \emph{collapse} of $T$, and the tree $T$ is called a \emph{refinement} of $S$.

The following is a classical observation, see \cite[Lemma~4.12]{GL-JSJ}. For a point $v\in T$, we denote by $G_v\leq G$ its stabiliser. Similarly, if $e\sq T$ is an edge, $G_e\leq G$ is its pointwise stabiliser.

\begin{lem}\label{lem:blow-up}
    Let $G\acts T$ be a splitting. 
    Consider a vertex $v\in T$ and a splitting $G_v\acts S$. Suppose that, for every edge $e\sq T$ incident to $v$, the stabiliser $G_e\leq G_v$ is elliptic in $S$. Then, there exists a splitting $G\acts T'$ with an equivariant collapse map $\pi\colon T'\ra T$ such that the preimage $\pi^{-1}(v)$ is $G_v$--equivariantly isomorphic to $S$, while the preimage $\pi^{-1}(w)$ is a singleton for all vertices $w\not\in G\cdot v$.
\end{lem}

Two $G$--trees are said to lie in the same \emph{deformation space} (in the sense of \cite{Forester,GL-JSJ}) if they have the same elliptic subgroups. The deformation space of a tree $G\acts T_1$ is said to \emph{dominate} the deformation space of another tree $G\acts T_2$ if every subgroup of $G$ that is elliptic in $T_1$ is also elliptic in $T_2$ (we also simply say that $T_1$ dominates $T_2$). 

We will need the following slight variation on \cite[Proposition~2.2]{GL-JSJ}. (Our assumptions here are slightly stronger, as we require edge-stabilisers of \emph{both} $T_1$ and $T_2$ to be elliptic, but this yields the stronger control on the edge-stabilisers of $T$ provided by Item~(3).)

\begin{lem}\label{lem:ref_dom}
    Let $G$ be finitely generated, and let $G\acts T_1$ and $G\acts T_2$ be two splittings. Suppose that the $G$--stabiliser of each edge of $T_1$ is elliptic in $T_2$, and that the $G$--stabiliser of each edge of $T_2$ is elliptic in $T_1$. Then there exists a splitting $G\acts T$ with the following properties:
    \begin{enumerate}
        \item $T$ is a refinement of $T_1$ that dominates $T_2$;
        \item a subgroup of $G$ is elliptic in $T$ if and only if it is elliptic in both $T_1$ and $T_2$;
        \item the $G$--stabiliser of each edge of $T$ is either the $G$--stabiliser of an edge of $T_1$ or $T_2$, or the intersection of an edge-stabiliser of $T_1$ with an edge-stabiliser of $T_2$;
        \item $T$ does not have any degree--$2$ vertices whose stabiliser fixes the two incident edges, except for any that $T_1$ might have already had and have not been blown up.
    \end{enumerate}
\end{lem}
\begin{proof}
    Since $G$ is finitely generated and acts minimally on $T_1$, there are only finitely many $G$--orbits of vertices in $T_1$. Let $v_1,\dots,v_k$ be representatives of these orbits, and let $G_i$ denote the $G$--stabiliser of $v_i$. If all $G_i$ are elliptic in $T_2$, then $T_1$ dominates $T_2$ and we can simply set $T:=T_1$. Thus, up to discarding some $v_i$ and re-indexing, we can suppose that each $G_i$ is non-elliptic in $T_2$. Let $M_i\sq T_2$ be the $G_i$--minimal subtree; this exists because $G_i$ is finitely generated relative to finitely many $G$--stabilisers of edges of $T_1$ \cite[Lemma~1.11]{Guir-Fourier}, and the latter are elliptic in $T_2$.

    Now, since $G$--stabilisers of edges of $T_1$ are elliptic in $T_2$, we can use \Cref{lem:blow-up} to blow up each $v_i\in T_1$ (and each of its $G$--translates) to a copy of $M_i\sq T_2$. Call $G\acts T$ the resulting refinement of $T_1$. It is clear that a subgroup of $G$ is elliptic in $T$ if and only if it is elliptic in both $T_1$ and $T_2$, proving properties~(1) and~(2). Regarding property~(4), consider the collapse map $\pi\colon T\ra T_1$. Note that every vertex of $M_i$ has degree $\geq 2$ within $M_i$. Thus, up to collapsing some edge orbits of $G_i\acts M_i$ without altering its deformation space, we can assume that, for every vertex $x\in\pi^{-1}(v_i)$ that has degree $2$ within $\pi^{-1}(v_i)$ and whose $G$--stabiliser fixes both incident edges, there exists an edge of $T\setminus\pi^{-1}(v_i)$ that is attached to $x$. We are only left to show property~(3).

    Consider an edge $e\sq T$. If $e$ projects to an edge of $T_1$, then its $G$--stabiliser $G_e$ is an edge-stabiliser of $T_1$. Suppose instead that $e$ gets collapsed to a vertex of $T_1$, without loss of generality to some $v_i$. Thus, there exists an edge $f\sq M_i\sq T_2$ such that $G_e$ equals the $G_i$--stabiliser of $f$; that is, $G_e=G_i\cap G_f$. By assumption, $G_f$ is elliptic in $T_1$. If $G_f$ fixes the vertex $v_i$, we have $G_f\leq G_i$ and hence $G_e=G_f$ is an edge-stabiliser of $T_2$. Otherwise, denoting by $g\sq T_1$ the edge incident to $v_i$ in the direction of $\Fix(G_f)\sq T_1$, we have $G_i\cap G_f=G_g\cap G_f$. Thus, $G_e$ is the intersection of an edge-stabiliser of $T_1$ and an edge-stabiliser of $T_2$, as required.
\end{proof}

The following is a first example of how we will exploit accessibility.

\begin{lem}\label{lem:defspaces_stabilise}
    Let $G$ be finitely generated. Consider a family of subgroups $\mc{F}$ such that $\mc{F}$ is closed under finite intersections and $G$ is unconditionally accessible over $\mc{F}$. 
    \begin{enumerate}
        \item For any collection $\{G\acts T_i\}_{i\in I}$ of $(\mc{F},\mc{F})$--trees, there exists an $(\mc{F},\mc{F})$--tree $G\acts T$ whose elliptic subgroups are precisely those that are elliptic in all $T_i$. Moreover, $T$ can be chosen to refine any of the $T_i$.
        \item There is a uniform bound to the length of any sequence $\mf{D}_0,\mf{D}_1,\mf{D}_2,\dots$ of pairwise distinct deformation spaces of $(\mc{F},\mc{F})$--trees such that $\mf{D}_i$ dominates $\mf{D}_j$ for $i>j$.
    \end{enumerate}
\end{lem}
\begin{proof}
    Consider a sequence $G\acts T_n$ of $(\mc{F},\mc{F})$--trees. Since $\mc{F}$ is intersection-closed, \Cref{lem:ref_dom} yields a sequence $G\acts T_n'$ of $(\mc{F},\mc{F})$--trees such that $T_0'=T_0$, each $T_{n+1}'$ is a refinement of $T_n'$, and a subgroup of $G$ is elliptic in $T_n'$ if and only if it is elliptic in all of $T_0,\dots,T_n$. Moreover, the $T_n'$ all have a bounded number of orbits of degree--$2$ vertices whose stabiliser fixes the two incident edges. Since $G$ is unconditionally accessible over $\mc{F}$, there is a uniform bound on the number of edge orbits of the $T_n'$, and hence the sequence of refinements $T_n'$ eventually stabilises. This proves part~(1).

    The proof of part~(2) is similar. Suppose that $T_0,\dots,T_k$ are $(\mc{F},\mc{F})$--trees such that $T_i$ dominates $T_j$ for $i>j$, and such that their deformation spaces are pairwise distinct. Up to collapsing edge orbits of $G\acts T_0$ without altering the deformation space of $T_0$, we can assume that $T_0$ has no degree--$2$ vertices whose stabiliser fixes the two incident edges (or that $T_0$ is a line on which $G$ acts vertex-transitively). Using \Cref{lem:ref_dom} as in the previous paragraph, we obtain $(\mc{F},\mc{F})$--trees $T_i'$ that refine $T_0$ and lie in the deformation space of $T_i$. Each $T_{i+1}'$ is a proper refinement of $T_i'$, as the deformation spaces of $T_i$ and $T_{i+1}$ are distinct. The tree $T_k'$ is then an $(\mc{F},\mc{F})$--tree with at least $k+1$ edge orbits and at most one orbit of degree--$2$ vertices whose stabiliser fixes the two incident edges. Thus, accessibility sets a uniform bound to the index $k$, proving part~(2).
\end{proof}

\section{The main theorem}\label{app:acc_R_trees}

In this section we prove \Cref{thm:acc_implies_nice}, which implies \Cref{thmintro} from the Introduction. Very roughly, starting from an action on an $\R$--tree $G\acts T$, our goal is to encode more and more of the $G$--action on $T$ into simplicial splittings $G\acts\Delta_n$ and to show that, because of accessibility, $T$ and $\Delta_n$ have the same elliptic subgroups for large $n$. We will do this by exploiting Rips--Sela theory, and we will often rely on the particularly clean treatment in \cite{Guir-Fourier} for reference.

We begin by giving a definition of quadratically hanging subgroups that is convenient for the following discussion. This is similar, but not identical, to the classical notion (see e.g.\ \cite[Definition~5.13]{GL-JSJ}). For a compact surface $\Sigma$, we say that a subgroup $P\leq\pi_1(\Sigma)$ is:
\begin{itemize}
    \item \emph{peripheral} if $P$ is conjugate to a subgroup of the fundamental group of a component of $\partial\Sigma$;
    \item a \emph{full peripheral} if either $P=\{1\}$ or $P$ is conjugate to the (entire) fundamental group of a component of $\partial\Sigma$.
\end{itemize}

\begin{defn}\label{defn:QH+}
    Let $G$ be a group, let $\mc{H}$ be a family of subgroups of $G$, and let $F\leq G$ be a subgroup. A subgroup $Q\leq G$ is an \emph{optimal quadratically hanging subgroup with fibre $F$}, relative to $\mc{H}$, if there exists a simplicial $G$--tree $G\acts S$ relative to $\mc{H}$ with the following properties:
    \begin{enumerate}
        \item there exists a vertex $x\in S$ such that $Q$ is the $G$--stabiliser of $x$;
        \item we have $F\lhd Q$ and an identification $Q/F\cong\pi_1(\Sigma)$ for a compact hyperbolic surface $\Sigma$ other than the pair of pants;
        \item for every edge $e\sq S$ incident to $x$, we have $F\lhd G_e$ and $G_e/F$ is a full peripheral in $\pi_1(\Sigma)$;
        \item for each component $B\sq\partial\Sigma$, there is at most one $Q$--orbit of edges $e\sq S$ incident to $x$ such that $G_e/F$ is conjugate to $\pi_1(B)$ in $\pi_1(\Sigma)$;
        \item for every $H\in\mc{H}$, the projection of $H\cap Q$ to $Q/F$ is peripheral in $\pi_1(\Sigma)$.
    \end{enumerate}
    If the above conditions hold, we also say that $x$ is an \emph{optimal QH vertex} of $S$. A subgroup $P\leq Q$ is a \emph{full peripheral} subgroup of $Q$, if we have $F\lhd P$ and $P/F$ is a full peripheral subgroup of $\pi_1(\Sigma)$. We say that the full peripheral subgroup $P$ is \emph{fake} if $P$ is not equal (hence also not commensurable) to the $G$--stabiliser of any edge of $S$ incident to $x$. A subgroup $E\leq Q$ is \emph{essential} if it is of the form $F\rtimes\langle\g\rangle$, where $\g$ is a lift of an essential simple closed curve on the surface $\Sigma$.
\end{defn}

Given another family $\mc{F}$ of subgroups of $G$, we denote by ${\rm QH}(\mc{F},\mc{H})$ the family of (optimal) quadratically hanging subgroups of $G$ with fibre in $\mc{F}$ relative to $\mc{H}$. We denote by ${\rm Per}(\mc{F},\mc{H})$ and ${\rm Ess}(\mc{F},\mc{H})$ the corresponding families of full peripheral and essential subgroups, respectively. 

We also denote by ${\rm AE}(\mc{F})$ the family of free abelian extensions of elements of $\mc{F}$, that is, the subgroups $E\leq G$ such that there exists $F\in\mc{F}$ with $F\lhd E$ and $E/F\cong\Z^k$ for some $k\geq 1$. Finally, recall that $\mc{F}_{\rm int}$ denotes the family of finite intersections of subgroups in $\mc{F}$.

Given an action on an $\R$--tree $G\acts T$, an arc $\beta\sq T$ is said to be \emph{stable} if all its sub-arcs have the same $G$--stabiliser. The action is \emph{BF-stable} if every arc of $T$ contains a stable sub-arc. We can finally state our main result.

\begin{thm}\label{thm:acc_implies_nice}
    Let $G$ be finitely presented and torsion-free. 
    Let $G\acts T$ be a minimal action on an $\R$--tree that is not a single point. Let $\mc{F}$ be the family of $G$--stabilisers of arcs of $T$, 
    and let $\mscr{E}$ be the family of subgroups of $G$ elliptic in $T$. Suppose that the following conditions are satisfied.
    \begin{enumerate}
        \item[(i)] The group $G$ is accessible over $\mc{F}_{\rm int}\cup{\rm Ess}(\mc{F},\emptyset)$.
        \item[(ii)] All elements of $\mc{F}_{\rm int}$ are finitely generated and root-closed.
        \item[(iii)] Chains of subgroups in $\mc{F}_{\rm int}\cup{\rm Ess}(\mc{F},\emptyset)$ have bounded length.
    \end{enumerate}
    Then all of the following hold.
    \begin{enumerate}
        \setlength\itemsep{.25em}
        \item There are only finitely many $G$--conjugacy classes of $G$--stabilisers of stable arcs of $T$.
        \item There are only finitely many $G$--orbits of branch points and of directions based at branch points. In particular, there are only finitely many $G$--orbits of points $p\in T$ such that $G_p\not\in\mc{F}$.
        \item For every point $p\in T$, all of the following hold.
        \begin{enumerate}
        \setlength\itemsep{.25em}
            \item The $G$--stabiliser $G_p$ is finitely generated.
            \item Suppose that $G_p\not\in\mc{F}$ and that $G_p$ is not a fake element of ${\rm Per}(\mc{F},\mscr{E})$. Then $G_p$ equals the $G$--stabiliser of a vertex $x$ in some $(\mc{F}_{\rm int}\cup{\rm Per}(\mc{F},\mscr{E}),\mscr{E})$--splitting $G\acts S$ with the following additional property: For every edge $e\sq S$ incident to $x$ such that $G_e\not\in\mc{F}_{\rm int}$, the vertex of $e$ other than $x$ is an optimal QH vertex of $S$ with fibre in $\mc{F}$ relative to $\mscr{E}$.
        \end{enumerate}
        \item Let a subgroup $H\leq G$ be non-elliptic in $T$. Then there is a $(\mc{F}_{\rm int}\cup{\rm Ess}(\mc{F},\mscr{E})\cup{\rm AE}(\mc{F}),\mscr{E})$--splitting of $G$ in which $H$ is non-elliptic. If $H$ does not preserve any line in $T$ acted upon indiscretely by its $G$--stabiliser, then $H$ is non-elliptic in a $(\mc{F}_{\rm int}\cup{\rm Ess}(\mc{F},\mscr{E}),\mscr{E})$--splitting.
    \end{enumerate}
\end{thm} 

Parts~(1) and~(2) of the theorem are fairly standard consequences of accessibility; we show them in Lemmas~\ref{lem:fin_many_stable} and~\ref{lem:branch_points}, respectively. Instead, finite generation of point-stabilisers and the construction of splittings relative to the whole family $\mscr{E}$ in part~(4) will take a bit more work.

\begin{rmk}\label{rmk:again_cc_near_QH}
    Suppose that, in part~(3) of \Cref{thm:acc_implies_nice}, we have $G_p\not\in\mc{F}$. Then one can deduce the following technical consequence from the statement of the theorem, which will be needed in the proof of \Cref{corintro}: 
    The stabiliser $G_p$ has a subgroup of index $\leq 2$ of the form $W_1\cap W_2$, where: 
    \begin{itemize}
        \item each $W_i$ is the $G_0$--stabiliser of a vertex in a splitting $G_0\acts S_i$, for a certain subgroup $G_0\leq G$ of index $\leq 2$;
        \item all edge groups of $S_i$ are intersections with $G_0$ of elements of $\mc{F}_{\rm int}\cup{\rm Ess}(\mc{F},\mscr{E})$. 
    \end{itemize}
    If $G_p$ is not a fake element of ${\rm Per}(\mc{F},\mscr{E})$, then the above fact is shown by exploiting the particular structure of the splitting $G\acts S$ provided by part~(3b) of the theorem: passing to a subgroup $G_0\leq G$, we can assume that all QH vertex groups of $S$ correspond to orientable surfaces; on each of these surfaces, we can suitably pick two different essential multicurves, then refine $S$ in two different ways by splitting the QH vertex groups along the chosen multicurves (\Cref{lem:blow-up}); finally, the splittings $G_0\acts S_i$ are obtained from these refinements by collapsing any edges with stabiliser in ${\rm Per}(\mc{F},\mscr{E})\setminus\mc{F}_{\rm int}$. We refer to the proof of \cite[Lemma~4.19]{Fio11} for additional details on how to pick the multicurves so that $W_1\cap W_2=G_0\cap G_p$.
    
    When $G_p$ happens to be a fake element of ${\rm Per}(\mc{F},\mscr{E})$, we can still rely on an almost identical argument: one instead uses the fact that $G_p$ corresponds to a boundary component of the associated surface that is ``not attached to anything''. We can then split this surface along two multicurves as above to again prove the statement.
\end{rmk}

We now start working towards the proof of \Cref{thm:acc_implies_nice}, which will occupy the rest of the section. Since $G$ is finitely presented, there exists a sequence $G\acts\mc{G}_n$ of (minimal) geometric $\R$--trees converging strongly to $T$ (in the sense of \cite{LP-geom}).
In particular, this means that there are $G$--equivariant morphisms $f_n\colon\mc{G}_n\ra T$ (continuous maps such that each arc of $\mc{G}_n$ can be subdivided into finitely many arcs on which $f_n$ is isometric), and also equivariant morphisms $f_n^m\colon\mc{G}_n\ra\mc{G}_m$ for $m>n$ such that $f_n=f_m\o f^m_n$. We denote by $\mscr{E}_n$ the family of subgroups of $G$ that are elliptic in $\mc{G}_n$, and observe that $\mscr{E}_n\sq\mscr{E}$ and $\mscr{E}_n\sq\mscr{E}_m$ for $m>n$. Strong convergence implies that, in particular, any finitely generated element of $\mscr{E}$ eventually lies in $\mscr{E}_n$.
    
Let $G\acts\mc{D}_n$ be the simplicial tree dual to the decomposition of $\mc{G}_n$ into indecomposable components and simplicial arcs \cite[Proposition~1.25]{Guir-Fourier}. The tree $\mc{D}_n$ is bipartite: there is a black vertex of $\mc{D}_n$ for each indecomposable component of $\mc{G}_n$, as well as for each maximal arc of $\mc{G}_n$ containing no branch points in its interior (a ``simplicial arc''); there is a white vertex of $\mc{D}_n$ for every point of intersection between subtrees of $\mc{G}_n$ associated to black vertices; edges of $\mc{D}_n$ correspond to point-subtree inclusions. The tree $G\acts\mc{D}_n$ is again minimal and relative to $\mscr{E}_n$. 

If $U\sq\mc{G}_n$ is an indecomposable component, we denote by $G_U$ its $G$--stabiliser, and by $[U]$ the corresponding black vertex of $\mc{D}_n$. Note that, if $e\sq\mc{D}_n$ is an edge incident to $[U]$, then the $G$--stabiliser of $e$ is the $G_U$--stabiliser of a point of $U$. Also note that $G_U$ is finitely generated (this follows e.g.\ from the proof of \cite[Proposition~1.25]{Guir-Fourier} since the cut locus is compact).

Observe that the action $G\acts T$ is BF--stable because of Condition~$(iii)$. If $U\sq\mc{G}_n$ is an indecomposable component, it follows that the image $f_n(U)\sq T$ is a stable subtree (see \cite[Section~1.6]{Guir-Fourier}). If $K_U$ denotes the $G$--stabiliser of one/all arcs of $f_n(U)\sq T$, we have $G_U\leq N_G(K_U)$ and $K_U\in\mc{F}$. In general, arc-stabilisers of $U\sq\mc{G}_n$ are proper subgroups of $K_U$. We will say that $U$ is \emph{saturated} if $K_U$ fixes $U$ pointwise, that is, if $K_U$ is the kernel of the action $G_U\acts U$, as well as the $G$--stabiliser of every arc of $U$.

If $U\sq\mc{G}_n$ is saturated, we obtain an indecomposable action $G_U/K_U\acts U$ with trivial arc-stabilisers. The output of the Rips machine now shows that $U$ is of one of three possible types --- axial, exotic, or surface; see e.g.\ Remark~1.29 and Proposition~A.6 in \cite{Guir-Fourier}. We will use the following facts about these three types:
\begin{itemize}
    \setlength\itemsep{.25em}
    \item If $U$ is axial, then $U$ is isometric to $\R$ and $G_U/K_U$ is free abelian of rank $\geq 2$ (note that $G_U/K_U$ is torsion-free because $K_U\in\mc{F}$ is root-closed by Condition~$(ii)$). Moreover, $K_U$ is the $G$--stabiliser of all edges of $\mc{D}_n$ incident to the black vertex $[U]$.
    \item If $U$ is exotic, then $G_U/K_U$ admits a (simplicial) free splitting with the same elliptic subgroups as the action $G_U/K_U\acts U$ \cite[Proposition~7.2]{Guir-CMH}. In particular, the point-stabilisers of the action $G_U/K_U\acts U$ are finitely generated, because they are free factors of the finitely generated group $G_U/K_U$. Moreover, $G_U$ splits over $K_U$ relative to the $G$--stabilisers of the edges of $\mc{D}_n$ incident to the black vertex $[U]$. We can then use this splitting of $G_U$ to refine the splitting $G\acts\mc{D}_n$ (\Cref{lem:blow-up}), if we so desire.
    \item If $U$ is of surface type, then the action $G_U/K_U\acts U$ is dual to an arational measured foliation on a compact hyperbolic surface $\Sigma_U$ of which $G_U/K_U$ is the fundamental group. Non-trivial point-stabilisers of $G_U/K_U\acts U$ are maximal cyclic subgroups of $G_U/K_U$, namely the conjugates of the fundamental groups of the components of $\partial\Sigma_U$. The $G$--stabiliser of each edge of $\mc{D}_n$ incident to $[U]$ is either equal to $K_U$ or to the extension of $K_U$ by one of these cyclic subgroups. In particular, $[U]$ is an optimal quadratically hanging vertex of $\mc{D}_n$ with fibre $K_U$. Note that $\Sigma_U$ is indeed not a pair of pants, as it supports an arational measured foliation.

    Any essential simple closed curve $\g\sq\Sigma_U$ gives a splitting of $G_U/K_U$ over the maximal cyclic subgroup $C:=\langle\g\rangle$, relative to the fundamental groups of the components of $\partial\Sigma_U$. In turn, this gives a splitting of $G_U$ over the essential subgroup $K_U\rtimes\langle\g\rangle$, relative to the $G$--stabilisers of the edges of $\mc{D}_n$ incident to the black vertex $[U]$. Again, we can use this splitting of $G_U$ to refine $G\acts\mc{D}_n$, if we so desire.
\end{itemize}
Summing up, any saturated indecomposable component of exotic or surface type $U\sq\mc{G}_n$ 
determines a particular splitting of $G$ over an element of $\mc{F}$ or ${\rm Ess}(\mc{F},\mscr{E}_n)$, relative to $\mscr{E}_n$. This splitting is obtained by first refining $\mc{D}_n$ at the vertex $[U]$ and then collapsing (most) edges of $\mc{D}_n$.

We can now use accessibility to prove part~(1) of \Cref{thm:acc_implies_nice}.

\begin{lem}\label{lem:fin_many_stable}
    There are only finitely many $G$--conjugacy classes of $G$--stabilisers of stable arcs of $T$. Moreover, there exists an integer $N$ such that each tree $\mc{G}_n$ has at most $N$ $G$--orbits of saturated indecomposable components.
\end{lem}
\begin{proof}
    Using \Cref{lem:unconditional_vs_conditional_acc}, the combination of Conditions~$(i)$ and~$(iii)$ implies that $G$ is unconditionally accessible over $\mc{F}\cup{\rm Ess}(\mc{F},\emptyset)$. Thus, let $N'$ be the largest number of edge orbits in an irredundant splitting of $G$ over $\mc{F}\cup{\rm Ess}(\mc{F},\mscr{E}_n)$. 
    
    Suppose for the sake of contradiction that there exist $N'+1$ stable arcs $\overline\beta_0,\dots,\overline\beta_{N'}\sq T$ whose $G$--stabilisers $B_0,\dots,B_{N'}$ are pairwise not $G$--conjugate. Since each $B_i$ is finitely generated, by Condition~$(ii)$, there exists $n\in\N$ such that the stable arcs $\overline\beta_i\sq T$ all lift isometrically to (stable) arcs $\beta_i\sq\mc{G}_n$ with $B_i$ as their $G$--stabiliser (this follows from the definition of strong convergence). Each $\beta_i$ shares an arc with a subset $U_i\sq\mc{G}_n$ that is either an indecomposable component or a maximal simplicial arc. 
    
    Now, consider the splitting $G\acts\mc{D}_n$. Form a new splitting $G\acts\mc{D}_n'$ by blowing up, for each indecomposable $U_i$ of exotic or surface type, the vertex $[U_i]\in\mc{D}_n$ to a one-edge splitting of $G_{U_i}$ over an element of $\mc{F}\cup{\rm Ess}(\mc{F},\mscr{E}_n)$, as described above. Note that the $G$--stabiliser of each new edge of $\mc{D}_n'$ is properly contained in the $G$--stabiliser of both incident vertices. 
    Then form a third splitting $G\acts\mc{D}_n''$ by collapsing all $G$--orbits of edges of $\mc{D}_n'$, except for those created in the blow-up, and except for the $G$--orbit of one edge incident to each of the vertices $[U_i]$ of axial or simplicial type. The result is that $G\acts\mc{D}_n''$ is a splitting over $\mc{F}\cup{\rm Ess}(\mc{F},\mscr{E}_n)$ with exactly $N'+1$ orbits of edges. Moreover, $\mc{D}_n''$ is irredundant since the $B_i$ are pairwise not $G$--conjugate (unless $N'=0$ and $\mc{D}_n''\cong\R$). This violates accessibility, providing the required contradiction and proving the first half of the lemma with $N:=N'$.

    The second half is similar: Suppose that some $\mc{G}_n$ has $2N'+1$ saturated indecomposable components $U_0,\dots,U_{2N'}$ in pairwise distinct $G$--orbits. We again consider $\mc{D}_n$ and inflate it to a tree $\mc{D}_n'$ by blowing up each vertex $[U_i]$ of exotic or surface type to a one-edge splitting of $G_{U_i}$ over an element of $\mc{F}\cup{\rm Ess}(\mc{F},\mscr{E}_n)$. Then we form a splitting $\mc{D}_n''$ by collapsing all edges of $\mc{D}_n'$ coming from edges of $\mc{D}_n$, except for the $G$--orbit of one edge incident to each $[U_i]$ of axial type. Again, $G\acts\mc{D}_n''$ is a splitting over $\mc{F}\cup{\rm Ess}(\mc{F},\mscr{E}_n)$ with exactly $2N'+1$ orbits of edges. Any failure of $\mc{D}_n''$ to be irredundant comes from a degree--$2$ vertex whose two incident edges $e,f$ have the same $G$--stabiliser; this can only happen if the other vertices of $e$ and $f$ are, respectively, $G$--translates of some $[U_i],[U_j]$ with $i\neq j$ and both $U_i,U_j$ axial (with $G$--conjugate kernels). In particular, the edges $e$ and $f$ are in distinct $G$--orbits (as $\mc{D}_n''\not\cong\R$) and the problem is solved by collapsing the $G$--orbit of one of them. 
    Ultimately, we obtain an irredundant splitting of $G$ over $\mc{F}\cup{\rm Ess}(\mc{F},\mscr{E}_n)$ with at least $N'+1$ edge orbits, violating again accessibility. We can thus take $N:=2N'$. 
\end{proof}

Using the previous lemma and Guirardel's work in \cite[Appendix~A]{Guir-Fourier}, we can now conclude that the indecomposable components of the geometric approximations $\mc{G}_n$ stabilise for large $n$. 

\begin{lem}\label{lem:ind_stabilise}
    There exists $n_0\in\N$ such that all the following hold for all $n\geq n_0$.
    \begin{enumerate}
        \item Every indecomposable component $U\sq\mc{G}_n$ is saturated.
        \item For all $m>n$ and each indecomposable component $U\sq\mc{G}_n$, the image $f^m_n(U)\sq\mc{G}_m$ is an indecomposable component of $\mc{G}_m$, and the map $f^m_n|_U\colon U\ra f^m_n(U)$ is an isometry. Moreover, the $G$--stabilisers of $U$ and $f^m_n(U)$ coincide.
        \item For all $m>n$, each indecomposable component of $\mc{G}_m$ is the image under $f^m_n$ of a (unique) indecomposable component of $\mc{G}_n$.
    \end{enumerate}
\end{lem}
\begin{proof}
For all integers $m>n$ and each indecomposable component $U\sq\mc{G}_n$, the image $f^m_n(U)\sq\mc{G}_m$ is indecomposable and, therefore, there exists an indecomposable component $V_m\sq\mc{G}_m$ such that $f^m_n(U)\sq V_m$ and $G_U\leq G_{V_m}$ (see again \cite[Section~1.6]{Guir-Fourier}). We have already seen that the image $f_n(U)\sq T$ is indecomposable and that all its arcs are stable with the same $G$--stabiliser $K_U\in\mc{F}$, which is finitely generated by Condition~$(ii)$. It follows that, for large enough $m$, an arc of $f_n(U)\sq T$ lifts to an arc of $f^m_n(U)\sq\mc{G}_m$ with precisely $K_U$ as its $G$--stabiliser. As a consequence, the indecomposable component $V_m\sq\mc{G}_m$ is saturated. Summing up, for every $n$ and each indecomposable component $U\sq\mc{G}_n$, the image $f^m_n(U)$ is contained in a saturated indecomposable component $V_m\sq\mc{G}_m$ for all sufficiently large $m$ (depending on $n$ and $U$).

Now, suppose that $U\sq\mc{G}_n$ is saturated to begin with, so that we obtain an indecomposable action with trivial arc-stabilisers $G_U/K_U\acts U$. By \cite[Lemma~A.7]{Guir-Fourier}, there exists an integer $n_U\geq n$ such that, for all $m\geq k\geq n_U$, the map $f^m_k|_{V_k}\colon V_k\ra V_m$ is an isometry (the hypotheses in Guirardel's lemma are slightly different from ours, but the key point is the same, namely \cite[Theorem~A.11]{Guir-Fourier}). In addition, Guirardel's lemma shows that either $V_k$ is axial or $G_{V_k}=G_{V_m}$. In fact, in our situation, we cannot have\footnote{The situation for axial components is more delicate in Guirardel's paper because the geometric approximations are acted upon by different groups there, as they arise from an action on an $\R$--tree of a possibly non-finitely-presented group $G$. In our setting, axial components do not need to be treated as an exception.} a proper inclusion $G_{V_k}\lneq G_{V_m}$ even when $V_k$ is axial: indeed, $V_m$ would also be axial, so the quotient $G_{V_m}/K_U$ would be abelian, which would imply that $G_{V_k}\lhd G_{V_m}$; hence, looking at the action $G\acts\mc{G}_k$, it would follow that $G_{V_m}$ preserves the line $V_k$. 

Summing up, for each indecomposable component $U\sq\mc{G}_n$, there exists an integer $n_U$ such that parts~(1) and~(2) of the lemma hold for $m\geq n_U$, for the indecomposable components of the $\mc{G}_m$ containing $f^m_n(U)$. Since the number of $G$--orbits of saturated indecomposable components in the $\mc{G}_n$ is uniformly bounded by \Cref{lem:fin_many_stable}, we get a uniform upper bound to the required integers $n_U$, proving parts~(1) and~(2) of the lemma. Part~(3) then immediately follows from this.
\end{proof}

The following proves part~(2) of \Cref{thm:acc_implies_nice} (note that, since arc-stabilisers are root-closed by Condition~$(ii)$, the $G$--stabiliser of every non-branch point of $T$ lies in $\mc{F}$). 

\begin{lem}\label{lem:branch_points}
    There are only finitely many $G$--orbits of branch points $p\in T$. For each $p$, the $G$--stabiliser of $p$ acts cofinitely on the set of directions at $p$.
\end{lem}
\begin{proof}
    Let $N_1$ be the largest number of edge orbits in a reduced splitting of $G$ over $\mc{F}$, which exists by Condition~$(i)$. Let $n_0\in\N$ be as in \Cref{lem:ind_stabilise}. Lemmas~\ref{lem:fin_many_stable} and~\ref{lem:ind_stabilise} show that, for all $n\geq n_0$, the geometric tree $\mc{G}_n$ has $N_2$ orbits of indecomposable components, for a fixed integer $N_2$. Furthermore, it follows from \cite{Guir-CMH,Guirardel-complexity} that there exists an integer $N_3$ such that, for each indecomposable component $U\sq\mc{G}_n$, the action $G_U/K_U\acts U$ has at most $N_3$ orbits of directions at branch points, and at most $N_3$ orbits of points with nontrivial $(G_U/K_U)$--stabiliser. 
    Finally, observe that, for each indecomposable component $U\sq\mc{G}_n$, the vertex $[U]\in\mc{D}_n$ is incident to edges of $\mc{D}_n$ in at most $N_1+N_3$ distinct $G_U$--orbits. Indeed, such edges correspond to points of $U$ where we can exit $U$, and the $G_U$--stabiliser of these points either equals $K_U\in\mc{F}$ or projects to a nontrivial subgroup of $G_U/K_U$. There are at most $N_1$ $G_U$--orbits of the former points, by accessibility, 
    and at most $N_3$ $G_U$--orbits of the latter points, by the definition of $N_3$.

    Now, let $b_1,\dots,b_k\in T$ be branch points in distinct $G$--orbits. For each $i$, consider an integer $m_i\geq 3$ and connected components $\kappa_i^1,\dots,\kappa_i^{m_i}\sq T\setminus\{b_i\}$ in distinct orbits under the $G$--stabiliser of $b_i$. Choose arcs $\beta_i^j\sq T$ based at $b_i$ with $\beta_i^j\setminus\{b_i\}\sq\kappa_i^j$. Our goal is to bound the sum $\sum_im_i$.
    
    Choose $n\geq n_0$ such that the finite forest $\bigcup_{i,j}\beta_i^j$ lifts isometrically to $\mc{G}_n$, and such that each arc $\beta_i^j$ lifts to an arc $\overline\beta_i^j\sq\mc{G}_n$ with the same $G$--stabiliser as $\beta_i^j$ (hence lying in $\mc{F}$). Call $b_i'\in\overline\beta_i^j$ the lift of $b_i$. Note that the points $b_i'$ are in distinct $G$--orbits and, for each $i$, the directions of $\mc{G}_n$ represented by the arcs $\overline\beta_i^j$ are in distinct orbits under the $G$--stabiliser of $b_i'$. Up to shrinking the arcs $\beta_i^j$, we can assume that each $\overline\beta_i^j$ is contained either in an indecomposable component of $\mc{G}_n$ or in a simplicial arc of $\mc{G}_n$; we then say that $\overline\beta_i^j$ is of \emph{indecomposable/simplicial kind}.
    
    It is convenient to partition the set $\{\overline\beta_i^j\}_{i,j}$ into the following three classes to count its elements.
    \begin{enumerate}
    \setlength\itemsep{.25em}
        \item Arcs contained in indecomposable components and based at one of their branch points. There are at most $N_2N_3$ of these.
        \item Arcs of simplicial kind based at points of $\mc{G}_n$ that lie in at least three distinct simplicial arcs. By accessibility, there are at most $2N_1$ of these arcs.
        \item All remaining arcs (of simplicial or indecomposable kind). These are based at points $p\in\mc{G}_n$ that are shared either by at least two indecomposable components, or by an indecomposable component and at least one simplicial arc; there are at most $N_2(N_1+N_3)$ orbits of such points $p$ under the $G$--action. In addition, the arcs in this third family are either contained in an indecomposable component $U\ni p$ such that $p$ has degree $2$ in $U$, or they are one of at most $2$ simplicial arcs based at $p$. Thus, there are at most $2+2N_2$ such arcs at the point $p$. 
    \end{enumerate}
    In conclusion, we obtain 
    \[ \sum_im_i\leq 2N_1+N_2N_3 + (2+2N_2)N_2(N_1+N_3) ,\] 
    which is the desired (albeit non-sharp) uniform bound on the total number of $G$--orbits of directions at branch points of $T$.
\end{proof}

The rest of the section is concerned with the proof of finite generation of point-stabilisers of $T$. Up to discarding finitely many $\mc{G}_n$, we assume in the following discussion that $n_0=0$, for the integer $n_0$ provided by \Cref{lem:ind_stabilise}. Denoting by $\mc{F}_{\rm stab}\sq\mc{F}$ the family of $G$--stabilisers of stable arcs of $T$, and recalling that $\mc{F}_{\rm stab}$ consists of finitely many $G$--conjugacy classes of finitely generated subgroups by \Cref{lem:fin_many_stable} and Condition~$(ii)$, we can similarly assume that all elements of $\mc{F}_{\rm stab}$ are elliptic in all $\mc{G}_n$. Since $T$ is BF--stable, every element of $\mc{F}$ is contained in an element of $\mc{F}_{\rm stab}$, and it follows that $\mc{F}\sq\mscr{E}_n$ for all $n$.

Denote by ${\rm Ind}$ the family of subgroups of $G$ arising as stabilisers of indecomposable components of the $\mc{G}_n$; by Lemmas~\ref{lem:fin_many_stable} and~\ref{lem:ind_stabilise}, ${\rm Ind}$ also consists of finitely many $G$--conjugacy classes of subgroups. For each $I\in{\rm Ind}$, we denote by $\partial I$ the family of subgroups of $I$ arising as (entire) point-stabilisers in the action $I\acts U$, where $U$ is the indecomposable component associated to $I$ in one/all $\mc{G}_n$ (by \Cref{lem:ind_stabilise}, the value of $n$ plays no role). We denote by $K_I\lhd I$ the kernel of the action $I\acts U$ (also denoted $K_U$ above). We clearly have $K_I\in\mc{F}$ for all $I\in{\rm Ind}$. Finally, we write $\partial{\rm Ind}:=\bigcup_{I\in{\rm Ind}}\partial I$.

It is convenient to give a name to the following type of splittings of $G$.

\begin{defn}\label{defn:excellent}
    A splitting $G\acts\Delta$ is \emph{excellent} if it satisfies the following properties:
    \begin{enumerate}
        \item for each $I\in{\rm Ind}$, there exists a vertex $x_I\in\Delta$ such that the $G$--stabiliser of $x_I$ is $I$, the $G$--stabiliser of every edge of $\Delta$ incident to $x_I$ lies in $\partial I$, and at most one edge incident to $x_I$ has a given element of $\partial I\setminus\{K_I\}$ as its stabiliser;
        \item if an edge $e\sq\Delta$ is not incident to $x_I$ for any $I\in{\rm Ind}$, then the $G$--stabiliser of $e$ lies in $\mc{F}_{\rm int}$;
        \item all subgroups in $\mscr{E}$ are elliptic in $\Delta$;
        \item if a subgroup of $G$ is elliptic in $\Delta$, then it lies in $\mscr{E}$ or is contained in an element of ${\rm Ind}$.
    \end{enumerate}
\end{defn}

Our next goal is showing that $G$ indeed admits an excellent splitting. For this, we first need the following observation.

\begin{lem}\label{lem:int_in_int}
    Let $G\acts S$ be a splitting relative to $\mc{F}$ satifying Items~(1) and~(2) in \Cref{defn:excellent}. For any vertex $x\in S$ with $G_x\not\in{\rm Ind}$ and any subgroup $F\in\mc{F}_{\rm int}$, we have $G_x\cap F\in\mc{F}_{\rm int}$.
\end{lem}
\begin{proof}
    Set $\Om:=G_x\cap F$ for simplicity. Since $S$ is a splitting relative to $\mc{F}$, there exists a vertex $y\in S$ fixed by $F$. If $F\leq G_x$, then $\Om$ coincides with $F$ and lies in $\mc{F}_{\rm int}$. Suppose instead that $F\not\leq G_x$ in the rest of the proof. Thus, we have $y\neq x$ and, letting $e\sq S$ be the edge incident to $x$ in the direction of $y$, we have $\Om=G_e\cap F$. If $G_e\in\mc{F}_{\rm int}$, we again have $\Om\in\mc{F}_{\rm int}$. 
    
    Suppose instead that $G_e\not\in\mc{F}_{\rm int}$. Let $z$ be the vertex of $e$ other than $x$, and note that its $G$--stabiliser must be some $I\in{\rm Ind}$ (as $G_x\not\in{\rm Ind}$ by hypothesis). Let $I\acts U$ be the corresponding indecomposable component of the $\mc{G}_n$, and let $\overline e\in U$ be a point with $G_e$ as its $I$--stabiliser. If $y=z$, then $F\leq I$ and, since we have assumed that $\mc{F}\sq\mscr{E}_n$ for all $n$, there is a point of $U$ fixed by $F$; this point must be distinct from $\overline e$ (otherwise we would have $F\leq G_e\leq G_x$), and this implies that $\Om=G_e\cap F=K_I\cap F\in\mc{F}_{\rm int}$. If instead $y\neq z$, let $f\sq S$ be the edge incident to $z$ in the direction of $y$. Again, there exists a point $\overline f\in U$ with $G_f$ as its $I$--stabiliser, and hence $\Om=F\cap (G_e\cap G_f)=F\cap K_I\in\mc{F}_{\rm int}$, as we wanted.
\end{proof}

We can now prove that excellent splittings indeed exist, provided that the $\mc{G}_n$ are not single indecomposable components. We will quickly complete the proof of \Cref{thm:acc_implies_nice} once this is shown.

\begin{prop}\label{prop:excellent_exists} 
    If $G\not\in{\rm Ind}$, there exists an excellent splitting $G\acts\Delta$.
\end{prop}
\begin{proof}
    Let $\mscr{S}$ be the family of splittings of $G$ that satisfy Items~(1) and~(2) in \Cref{defn:excellent}. For each $k\in\N$, let $\mscr{S}_k\sq\mscr{S}$ be the subset of splittings relative to $\mscr{E}_k$. We have $\mscr{S}_k\supseteq\mscr{S}_m$ for $m>k$, as $\mscr{E}_k\sq\mscr{E}_m$. 

    We claim that $\mscr{S}_k\neq\emptyset$ for all $k\in\N$. If ${\rm Ind}\neq\emptyset$, we can for instance consider the tree $G\acts\mc{D}_k$ defined above, and collapse all $G$--orbits of edges of $\mc{D}_k$ that are not incident to any black vertex representing an indecomposable component of $\mc{G}_k$. (Note that $\mc{D}_k$ is not a single vertex because $G\not\in{\rm Ind}$.) 
    If instead ${\rm Ind}=\emptyset$, 
    we can consider any stable arc $\beta\sq T$ and choose $n>k$ so that $\mc{G}_n$ has a (necessarily simplicial) arc $\beta'$ with stabiliser $G_{\beta}\in\mc{F}$; we then collapse all connected components of $\mc{G}_n\setminus\bigcup_{g\in G}g\beta'$. Either way, the collapse lies in $\mscr{S}_k$, showing that $\mscr{S}_k\neq\emptyset$.
    
    In order to progress with the proof of the proposition, we will need the following observation.

    \smallskip
    {\bf Claim~1.} \emph{For any two splittings $S_1,S_2\in\mscr{S}_k$, there exists a third splitting $S\in\mscr{S}_k$ that refines $S_1$ and dominates $S_2$.}

    \smallskip\noindent
    \emph{Proof of Claim~1.}
    This essentially follows from \Cref{lem:ref_dom}, except that we need to check that the refinement of $S_1$ constructed there truly lies in $\mscr{S}$. For this, consider a vertex $x\in S_1$ such that $G_x$ is not elliptic in $S_2$, and note that this implies that $G_x\not\in{\rm Ind}$. Also note that $G_x$ is finitely generated because $S_1$ has finitely generated edge-stabilisers. Let $M\sq S_2$ be the $G_x$--minimal subtree. We claim that there exists an edge of $M$ whose $G_x$--stabiliser lies in $\mc{F}_{\rm int}$.

    If there exists an edge $e\sq M$ such that $G_e\in\mc{F}_{\rm int}$, then the $G_x$--stabiliser of $e$ lies in $\mc{F}_{\rm int}$ by \Cref{lem:int_in_int}.
    Thus, we can suppose that no edge of $M$ has $G$--stabiliser in $\mc{F}_{\rm int}$. Considering any edge $e\sq M$, this implies that there exists a vertex $y\in e$ whose $G$--stabiliser is some $I\in{\rm Ind}$. Note that there also exists an edge $f\sq M$ incident to $y$ with $f\neq e$. Let $I\acts U$ be the indecomposable component associated to $I$. If $y'$ is the vertex of $S_1$ fixed by $I$, we have $y'\neq x$, as we have seen above that $G_x\not\in{\rm Ind}$. Now, let $g\sq S_1$ be the edge incident to $y'$ in the direction of $x$. The $G$--stabilisers of the edges $e,f,g$ equal the $I$--stabilisers of three points of $U$, and at least one of the two subgroups $G_e$ and $G_f$ must be different from $G_g$, unless they all equal $K_I$ (recall Item~(1) in \Cref{defn:excellent}). Thus, assuming without loss of generality that $G_e\neq G_g$, we have $G_e\cap G_g=K_I\in\mc{F}$. It follows that the $G_x$--stabiliser of the edge $e\sq M$ equals $G_x\cap (G_e\cap G_g)\in\mc{F}_{\rm int}$ by \Cref{lem:int_in_int}.

    In conclusion, we have shown that there always exists an edge $e\sq M$ whose $G_x$--stabiliser lies in $\mc{F}_{\rm int}$. Collapsing all other $G_x$--orbits of edges of $M$, we obtain a one-edge splitting of $G_x\acts M'$. We can then blow up the vertex $x\in S_1$ to a copy of $M'$, thus obtaining another splitting of $G$ in $\mscr{S}_k$. We can only repeat this procedure a finite number of times, as $G$ is unconditionally accessible over $\mc{F}_{\rm int}$ by \Cref{lem:unconditional_vs_conditional_acc} and Conditions~$(i)$ and~$(iii)$. Thus, we eventually obtain a refinement of $S_1$ that lies in $\mscr{S}_k$ and dominates $S_2$ (see the proof of \Cref{lem:ref_dom} for details). 
    \hfill$\blacksquare$

    \smallskip
    Now observe that, for each $k\in\N$, there exists a splitting $\Delta_k\in\mscr{S}_k$ that dominates all other splittings in $\mscr{S}_k$. This is proved exactly as in \Cref{lem:defspaces_stabilise}(1), using the fact that $G$ is unconditionally accessible over $\mc{F}_{\rm int}$, and using Claim~1 in place of \Cref{lem:ref_dom} (we cannot directly use \Cref{lem:defspaces_stabilise}, as some edge-stabilisers of the splittings in $\mscr{S}$ do not lie in $\mc{F}_{\rm int}$).

    We now show that, for each $k\in\N$, the splitting $G\acts\Delta_k$ satisfies Item~(4) in \Cref{defn:excellent}.

    \smallskip
    {\bf Claim~2.} \emph{If $x\in\Delta_k$ is a vertex such that $G_x\not\in{\rm Ind}$, then $G_x$ is elliptic in the $\R$--tree $T$.}

    \smallskip\noindent
    \emph{Proof of Claim~2.}
    Observe that the $G$--stabilisers of the edges of $\Delta_k$ incident to $x$ lie in $\partial{\rm Ind}\cup\mc{F}_{\rm int}$; these subgroups are elliptic in $T$ and in $\mc{G}_n$ for all $n$. Suppose for the sake of contradiction that $G_x$ is not elliptic in $T$. 
    
    We will show that, for all sufficiently large integers $n$, the group $G_x$ splits over an element of $\mc{F}_{\rm int}$ that it contains, relative to $\mscr{E}_n$. This will imply that $G\acts\Delta_k$ can be refined (using \Cref{lem:blow-up}) into a splitting that still lies in $\mscr{S}_k$ and in which $G_x$ is no longer elliptic, which will contradict the fact that $\Delta_k$ dominates all splittings in $\mscr{S}_k$.

    Now, note that $G_x$ is finitely generated relative to finitely many elements of $\partial{\rm Ind}\cup\mc{F}_{\rm int}$ by \cite[Lemma~1.11]{Guir-Fourier}. 
    Thus, the fact that $G_x$ is not elliptic in $T$ implies that $G_x$ contains an element $g\in G_x$ that is loxodromic in $T$, and hence loxodromic in $\mc{G}_n$ for all $n$. Let $\alpha\sq T$ and $\alpha_n\sq\mc{G}_n$ be the axes of $g$, and choose a stable arc $\beta\sq\alpha$. For large $n$, the arc $\beta$ lifts to an arc $\beta_n\sq\alpha_n$ with the same $G$--stabiliser as $\beta$ (as the elements of $\mc{F}$ are finitely generated). Up to modifying $\beta_n$ we can assume that either $\beta_n$ is a maximal simplicial arc of $\mc{G}_n$ or $\beta_n$ is an arc of some indecomposable component $U\sq\mc{G}_n$. In the former case, we can collapse to a point every arc of $\mc{G}_n$ whose $G$--translates are all disjoint from the interior of $\beta_n$, thus obtaining a one-edge splitting of $G$ over the conjugates of $G_{\beta}\in\mc{F}$, in which the elements of $\mscr{E}_n$ are elliptic and $G_x$ is not. As $G_x\cap G_{\beta}\in\mc{F}_{\rm int}$ by \Cref{lem:int_in_int}, this yields the desired splitting of $G_x$ in this case.

    Thus, suppose that the axis $\alpha_n\sq\mc{G}_n$ shares an arc with an indecomposable component $U\sq\mc{G}_n$. Note that the intersection $G_x\cap G_U$ is elliptic in all trees $\mc{G}_m$ (as $x$ is a vertex of a splitting in $\mscr{S}$). In particular, there is at most one point $z\in U$ such that the intersection between $G_x$ and the $G_U$--stabiliser of $z$ is not contained in the kernel $K_U$. Now, consider the action $G\acts\mc{D}_n$ and let $e,f\sq\mc{D}_n$ be the two edges incident to $[U]$ corresponding to the two points where $\alpha_n$ exits $U$. By the previous discussion, up to swapping $e$ and $f$, we can assume that $G_x\cap G_e\leq K_U\in\mc{F}_{\rm int}$. It follows that, $G_x\cap G_e=G_x\cap K_U\in\mc{F}_{\rm int}$ by \Cref{lem:int_in_int}. In conclusion, after collapsing all edges of $\mc{D}_n$ outside the $G$--orbit of $e$, we again obtain a splitting of $G$ over an element of $\mc{F}_{\rm int}$, in which the elements of $\mscr{E}_n$ are elliptic and $G_x$ is not. This induces the desired splitting of $G_x$ and, as discussed at the start, it proves the claim.
    \hfill$\blacksquare$

    \smallskip
    Finally, observe that the splitting $G\acts\Delta_k$ dominates the splitting $G\acts\Delta_m$ for all $m>k$, as we have $\mscr{S}_k\supseteq\mscr{S}_m$. Using again Claim~1 and the fact that $G$ is unconditionally accessible over $\mc{F}_{\rm int}$, as in the proof of \Cref{lem:defspaces_stabilise}(2), we conclude that the deformation spaces of the trees $\Delta_n$ must stabilise for large $n$.

    In conclusion, setting $\Delta:=\Delta_k$ for a sufficiently large value of $k$, we obtain a splitting that satisfies Items~(1),~(2) and~(4) of \Cref{defn:excellent} and the following slightly weaker form of Item~(3): all subgroups in the union $\bigcup_{n\in\N}\mscr{E}_n$ are elliptic in $\Delta$.

    In order to conclude that $\Delta$ is indeed excellent, consider some $E\in\mscr{E}$. Each finitely generated subgroup of $E$ lies in $\mscr{E}_n$ for large $n$, so all finitely generated subgroups of $E$ are elliptic in $\Delta$. Since chains in $\mc{F}_{\rm int}$ are bounded by Condition~$(iii)$, it follows that $E$ is itself elliptic in $\Delta$ (see \Cref{lem:FK}). This shows that $G\acts\Delta$ is excellent, proving the proposition.
\end{proof}

Before we continue, it is important to observe that an excellent splitting can be modified (near the vertices fixed by exotic elements of ${\rm Ind}$) so as to make its edge groups finitely generated.

\begin{rmk}\label{rmk:excellent->fg_edges}
    Let $G\acts\Delta$ be an excellent splitting. Consider an edge $e\sq\Delta$. If $G_e\in\mc{F}_{\rm int}$, then $G_e$ is finitely generated by Condition~$(ii)$. If $G_e\not\in\mc{F}_{\rm int}$, then there exists a vertex $y\in e$ such that $G_y$ equals some $I\in{\rm Ind}$ and $G_e\in\partial I$. If $I$ is of axial type, then $G_e$ equals the kernel $K_I\in\mc{F}$, which is finitely generated. Similarly, if $I$ is of surface type, then $G_e\in{\rm Per}(\mc{F},\mscr{E})$, so $G_e$ is an extension of $K_I$ by a cyclic subgroup and it is again finitely generated. 
    
    Suppose instead that $I$ is of exotic type. As mentioned above, $I$ admits a splitting over $K_I$ whose vertex groups are precisely the elements of $\partial I$. 
    We can thus blow up the vertex $y\in\Delta$ to a copy of this splitting of $I$, then collapse all edges that were incident to $y$ in $\Delta$ (and their $G$--translates). This procedure does not affect the stabilisers of the vertices of $\Delta\setminus(G\cdot y)$.

    Summing up, there exists a $(\mc{F}_{\rm int}\cup{\rm Per}(\mc{F},\mscr{E}),\mscr{E})$--splitting $G\acts\Delta'$ whose vertex-stabilisers are the elements of ${\rm Ind}$ not of exotic type, the kernels of the exotic components, and the vertex-stabilisers of $\Delta$ outside ${\rm Ind}$. All edge-stabilisers of $\Delta'$ are finitely generated, and hence all vertex-stabilisers of $\Delta'$ are finitely generated as well. Moreover, edges of $\Delta'$ with stabiliser in ${\rm Per}(\mc{F},\mscr{E})\setminus\mc{F}_{\rm int}$ have a vertex whose stabiliser is an element of ${\rm Ind}$ of surface type, which in particular lies in ${\rm QH}(\mc{F},\mscr{E})$.
\end{rmk}

We can finally complete the proof of \Cref{thm:acc_implies_nice}.

\begin{proof}[Proof of \Cref{thm:acc_implies_nice}]
    We can assume throughout the proof that $G\not\in{\rm Ind}$. Otherwise, all $\mc{G}_n$ are indecomposable, and they are equivariantly isometric to each other and to $T$. In that case, all parts of the theorem follow from the above discussion of indecomposable actions. 
    
    Parts~(1) and~(2) of the theorem were shown above in Lemmas~\ref{lem:fin_many_stable} and~\ref{lem:branch_points}, respectively. (Recall that, since arc-stabilisers are root-closed by Condition~$(ii)$, the $G$--stabiliser of every non-branch point of $T$ lies in $\mc{F}$.) Thus, it suffices to discuss parts~(3) and (4).
    
    Part~(3) can be quickly deduced from \Cref{prop:excellent_exists}. Let $G\acts\Delta'$ be the modification of an excellent splitting, as constructed in \Cref{rmk:excellent->fg_edges}. Consider a point $p\in T$ and its stabiliser $G_p$. Since $G_p\in\mscr{E}$, there exists a vertex $x\in\Delta'$ such that $G_p\leq G_x$. If $G_x$ is some element $I\in{\rm Ind}$ and if $G_p$ fixes no other vertices of $\Delta'$, then $G_p$ equals the $I$--stabiliser of a point of the associated indecomposable component $U$, which is of axial or surface type. In this case, either $G_p\in\mc{F}$ or $G_p$ is a fake element of ${\rm Per}(\mc{F},\mscr{E})$, which are the two situations excluded from part~(3b) of the theorem, but they nevertheless guarantee that $G_p$ is finitely generated.

    Otherwise, we can assume that $G_x\not\in{\rm Ind}$. Thus, we have $G_x\in\mscr{E}$ and there exists a point $q\in T$ such that $G_p\leq G_x\leq G_q$. If $q\neq p$, then $G_p$ coincides with the $G$--stabiliser of the arc $[p,q]\sq T$, and hence we again have $G_p\in\mc{F}$. Finally, if $q=p$, we obtain the equality $G_p=G_x$ and, as observed in \Cref{rmk:excellent->fg_edges}, the stabiliser $G_p$ is indeed finitely generated and a vertex group of a splitting of $G$ with the properties listed in part~(3b). This proves part~(3) of the theorem.

    We finally address part~(4). Consider a subgroup $H\leq G$ that is non-elliptic in $T$. Suppose first that $H$ is elliptic in $G\acts\Delta'$, so that we have $H\leq I$ for some $I\in{\rm Ind}$ of surface or axial type. 
    
    If $I$ is of surface type, we can refine $\Delta'$ by blowing up the $I$--fixed vertex to a one-edge splitting of $I$ over an element of ${\rm Ess}(\mc{F},\mscr{E})$, relative to $\partial I$. If we then collapse all edges not created during the blow-up, we obtain a one-edge $({\rm Ess}(\mc{F},\mscr{E}),\mscr{E})$--splitting of $G$ in which $H$ is non-elliptic. 
    
    If $I$ is of axial type, we have $I/K_I\cong\Z^k$ where $k\geq 2$ and $K_I\in\mc{F}$. (Recall that, since $K_I$ is root-closed by Condition~$(ii)$, there are no reflections in the action of $I$ on its invariant line.) Choosing any subgroup $E\lhd I$ with $K_I\lhd E$ and $I/E\cong\Z$, we have $E\in{\rm AE}(\mc{F})$ and we obtain an HNN splitting of $I$ with kernel $E$. As in the previous paragraph, we can refine $\Delta'$ by blowing up the $I$--fixed vertex to this HNN splitting, and then collapse all edges not created during the blow-up. We thus obtain a one-edge $({\rm AE}(\mc{F}),\mscr{E})$--splitting of $G$ in which $H$ is not elliptic. Note that this paragraph is relevant only if $H$ preserves an indiscrete line in $T$, and otherwise it can be skipped.

    Finally, suppose that $H$ is non-elliptic in $\Delta'$. Note that $H$ contains elements that are loxodromic in $\Delta'$, as chains in $\mc{F}_{\rm int}$ have bounded length by Condition~$(iii)$ (using \Cref{lem:FK}). Thus, $H$ has a well-defined minimal subtree $M\sq\Delta'$. If the $G$--stabiliser of at least one edge of $M$ lies in $\mc{F}_{\rm int}$, then we can collapse all edges of $\Delta'$ outside the $G$--orbit of this edge, and thus obtain a $(\mc{F}_{\rm int},\mscr{E})$--splitting of $G$ in which $H$ is not elliptic. Otherwise, there exists a vertex $x\in M$ such that $G_x$ is an element of ${\rm Ind}$ of surface type. In this case, we can again refine $\Delta'$ by blowing up the vertex $x$ to a one-edge splitting of $G_x$ over an element of ${\rm Ess}(\mc{F},\mscr{E})$, relative to $\partial G_x$. Collapsing all edges not created during the blow-up, we again obtain a one-edge $({\rm Ess}(\mc{F},\mscr{E}),\mscr{E})$--splitting of $G$ in which $H$ is non elliptic. This proves part~(4) and concludes the proof of \Cref{thm:acc_implies_nice}.
\end{proof}

\section{Special groups}\label{sect:special}

In \Cref{sub:special_main}, we deduce \Cref{corintro} from \Cref{thm:acc_implies_nice}. First though, we need to review some basic facts on convex-cocompact subgroups.

\subsection{Convex-cocompactness}\label{sub:cc}

A \emph{cocompact cubulation} $Q\acts X$ is a proper cocompact action of a discrete group on a CAT(0) cube complex. Given a cocompact cubulation, a subgroup $H\leq Q$ is \emph{convex-cocompact} if there exists an $H$--invariant convex subcomplex $C\sq X$ such that the action $H\acts X$ is cocompact. We refer to \cite[Lemma~3.2]{Fio10a} for equivalent conditions.

In the presence of a (simplicial) splitting of $Q$, convex-cocompactness can be transferred from edge-stabilisers to vertex-stabilisers:

\begin{prop}\label{prop:cc_edge->vertex}
    Let $Q\acts X$ be a cocompact cubulation. If $Q\acts T$ is a splitting such that edge-stabilisers are convex-cocompact in $X$, then vertex-stabilisers are convex-cocompact in $X$.
\end{prop}
\begin{proof}
    We start by constructing a cocompact cubulation $Q\acts Y$ having both actions $Q\acts X$ and $Q\acts T$ as \emph{restriction quotients} in the sense of \cite[p.\ 860]{CS11} (also see \cite[Section~2.5.1]{Fio10a} for details); in other words, $X$ and $T$ can both be obtained from $Y$ by collapsing some $Q$--orbits of hyperplanes. Such a complex $Y$ can be obtained by applying \cite[Proposition~7.9]{Fio10e} exactly as in the proof of \cite[Theorem~2.17, (3)$\Ra$(4)]{FLS}. Omitting some details, one considers the diagonal action $Q\acts X\x T$ and finds a $Q$--invariant, $Q$--cofinite median subalgebra $M\sq X^{(0)}\x T^{(0)}$ that is mapped onto $X^{(0)}$ by the factor projection $X\x T\ra X$. Using Chepoi--Roller duality \cite{Chepoi,Roller}, the action $Q\acts M$ is the action on the $0$--skeleton of the required CAT(0) cube complex $Y$.

    Now that we have $Y$, there is by construction a $Q$--invariant set $\mc{W}$ of pairwise-disjoint hyperplanes of $Y$ such that its dual tree is $Q$--equivariantly isomorphic to $T$. Let $Y^{\circ}$ be a copy of $Y$ from which we have removed the interiors of the carriers of the hyperplanes in $\mc{W}$ (that is, all open cubes intersecting elements of $\mc{W}$). The $Q$--stabilisers of the vertices of $T$ are precisely the $Q$--stabilisers of the connected components of $Y^{\circ}$. Each of these components is a convex subcomplex of $Y$ and it is acted upon cocompactly by its stabiliser, since $Q\acts Y$ is cocompact and $Y^{\circ}\sq Y$ is $Q$--invariant. This shows that $Q$--stabilisers of vertices of $T$ are convex-cocompact in $Y$.

    Finally, the equivariant projection $Y\ra X$ takes convex subcomplexes to convex subcomplexes, so any subgroup of $Q$ that is convex-cocompact in $Y$ is also convex-cocompact in $X$.
\end{proof}

Recall from the Introduction that, for a group $G$, the family $\mc{Z}(G)$ of centralisers is defined as
\[ \mc{Z}(G):=\{Z_G(A)\mid A\sq G\}. \]
Equivalently, the elements of $\mc{Z}(G)$ are precisely those subgroups $H\leq G$ that satisfy $Z_GZ_G(H)=H$.

When we speak of convex-cocompactness for subgroups of a right-angled Artin group $A_{\G}$, this is always meant with respect to the standard $A_{\G}$--action on the universal cover of the Salvetti complex. A group $G$ is \emph{special} \cite{HW08} if there exists a convex-cocompact embedding $\iota\colon G\hookrightarrow A_{\G}$ for some right-angled Artin group $A_{\G}$, meaning that $\iota(G)\leq A_{\G}$ is convex-cocompact.

If $G$ is special, then each choice of a convex-cocompact embedding $\iota\colon G\hookrightarrow A_{\G}$ induces a notion of convex-cocompactness for subgroups of $G$: we say that $H\leq G$ is \emph{convex-cocompact} (with respect to $\iota$) if the subgroup $\iota(H)$ is convex-cocompact in $A_{\G}$. Different choices of $\iota$ induce different notions of convex-cocompactness on $G$ (see \cite{FLS}), but all have the following properties.

\begin{lem}\label{lem:cc_basics}
    Let $G\leq A_{\G}$ and $H,K\leq G$ be convex-cocompact.
    \begin{enumerate}
        \item The intersection $H\cap K$ is convex-cocompact.
        \item The normaliser $N_G(H)$ is convex-cocompact and virtually splits as $H\x P$ for some $P\leq G$.
        \item The set $\{H\cap gKg^{-1}\mid g\in G\}$ is finite up to $H$--conjugacy.
        \item All centralisers $Z\in\mc{Z}(G)$ are convex-cocompact and root-closed.
    \end{enumerate}
\end{lem}
\begin{proof}
     Parts~(1),~(2) and~(4) are, respectively, Lemma~2.8, Corollary~2.13(1) and Remark~3.30 in \cite{Fio10e}. As to part~(3), let $\X_{\G}$ be the universal cover of the Salvetti complex of $A_{\G}$, and let $C\sq\X_{\G}$ be a convex subcomplex on which $H$ acts cocompactly. There is an integer $N$ such that, for each $g\in G$, the intersection $H\cap gKg^{-1}$ acts on a convex subcomplex of $C$ with at most $N$ orbits of vertices: this was shown in \cite[Lemma~7.5]{Fio10e} when $H=K$, and the general case is identical. Now, the group $H$ has only finitely many conjugacy classes of subgroups acting on convex subcomplexes of $C$ with at most $N$ orbits of vertices (see the claim in the proof of \cite[Lemma~2.9]{Fio10e}), and this yields part~(3).
\end{proof}

\begin{rmk}\label{rmk:centraliser_chains}
    If $G$ special, then chains of subgroups in the collection of centralisers $\mc{Z}(G)$ have uniformly bounded length. For instance, this is a consequence of \cite[Lemma~3.36]{Fio10e}.
\end{rmk}

\subsection{Special groups and real trees}\label{sub:special_main}

\Cref{cor:acc_implies_nice} below implies \Cref{corintro} from the Introduction. First, we need one last lemma.

\begin{lem}\label{lem:Ess_Z}
    If $G$ is a special group, then ${\rm Ess}(\mc{Z}(G),\emptyset)\sq\mc{Z}(G)$.
\end{lem}
\begin{proof}
    Let $G\acts S$ be a splitting with a vertex $x\in S$ whose stabiliser $G_x$ is an (optimal) quadratically hanging subgroup with fibre $F\in\mc{Z}(G)$. Let $\gamma\in G_x$ be an element whose projection $\overline\gamma\in G_x/F$ represents an essential simple closed curve on the associated surface. We need to show that the subgroup $E:=F\rtimes\langle\gamma\rangle \leq G_x$ lies in $\mc{Z}(G)$.
    
    Since $F$ lies in $\mc{Z}(G)$, it is convex-cocompact in $G$ and so \Cref{lem:cc_basics}(2) guarantees that $N_G(F)$ virtually splits as $F\x P$ for a subgroup $P\leq G$. In particular, $E$ virtually splits as $F\x\langle p\rangle$ for some element $p\in P$ projecting to a power of $\overline\g\in G_x/F$. Note that, since $\overline\gamma$ is essential, the vertex $x$ is the only fixed point of $p$ in the tree $S$. This implies that $Z_G(p)\leq G_x$ and hence $Z_G(p)\leq E$.
    
    In conclusion, we have $F\lhd Z_G(p)\leq E$ and $Z_G(p)/F=\langle\overline\g^k\rangle$ for some $k\geq 1$. It follows that $Z_G(p)$ is a finite-index subgroup of $E$ and, since $Z_G(p)$ is root-closed by \Cref{lem:cc_basics}(4), we must have $E=Z_G(p)$. This shows that $E\in\mc{Z}(G)$ as desired. 
\end{proof}

Recall that $\mc{VP}(G)$ denotes the family of subgroups of $G$ that virtually split as the direct product of two infinite groups.

\begin{cor}\label{cor:acc_implies_nice}
    Let $G$ be a special group. Suppose that $G$ is accessible over $\mc{Z}(G)$. Let $G\acts T$ be a minimal $\R$--tree with arc-stabilisers in $\mc{Z}(G)$. Let $\mscr{E}$ be the family of elliptic subgroups of $G$.
    \begin{enumerate}
        \item For every $p\in T$, the stabiliser $G_p$ is convex-cocompact.
        \item There are only finitely many $G$--conjugacy classes of point-stabilisers of $T$.
        \item Let a subgroup $H\leq G$ be non-elliptic in $T$. Suppose either that $\mc{VP}(G)\sq\mscr{E}$, or that $H$ does not preserve any line in $T$. Then $H$ is non-elliptic in a $(\mc{Z}(G),\mscr{E})$--splitting of $G$.
        \item For every point $p\in T$, one of the following holds:
            \begin{enumerate}
                \item $G_p$ is trivial;
                \item $G_p\in\mc{Z}(G)$ and $G_p\cong\Z$;
                \item $G_p$ is contained in a subgroup in $\mc{VP}(G)$;
                \item $G_p$ is a vertex group in a $(\mc{Z}(G)\cup{\rm Per}(\mc{Z}(G),\mscr{E}),\mscr{E})$--splitting of $G$.
            \end{enumerate}
    \end{enumerate}
\end{cor}
\begin{proof}
    Since ${\rm Ess}(\mc{Z}(G),\emptyset)\sq\mc{Z}(G)$ by \Cref{lem:Ess_Z}, we have that the action $G\acts T$ satisfies the hypotheses of \Cref{thm:acc_implies_nice}. Indeed, Condition~$(i)$ holds because $G$ is accessible over $\mc{Z}(G)$, while Conditions~$(ii)$ and~$(iii)$ follow from \Cref{lem:cc_basics}(4) and \Cref{rmk:centraliser_chains}.
    
    Now, part~(3) of the corollary follows from \Cref{thm:acc_implies_nice}(4) if we can show that there are no indiscrete lines in $T$ when the elements of $\mc{VP}(G)$ are elliptic. Thus, consider a line $\alpha\sq T$, let $G_{\alpha}$ be the subgroup of $G$ preserving it, and suppose that the action $G_{\alpha}\acts\alpha$ is indiscrete. The kernel $K_{\alpha}$ of the action $G_{\alpha}\acts\alpha$ is an intersection of arc-stabilisers of $T$ and so it lies in $\mc{Z}(G)$. Thus, $K_{\alpha}$ is convex-cocompact and \Cref{lem:cc_basics}(2) implies that $G_{\alpha}$ virtually splits as $K_{\alpha}\x A$ with $A$ free abelian. The rank of $A$ is $\geq 2$ because of the indiscreteness of $\alpha$. Thus, $G_{\alpha}$ lies in $\mc{VP}(G)$ and is non-elliptic in $T$, as desired.
    
    Part~(1) of the corollary follows from \Cref{thm:acc_implies_nice}(3b) via the consequence outlined in \Cref{rmk:again_cc_near_QH}. Indeed, if $p\in T$ and $G_p\not\in\mc{Z}(G)$, then the description of $G_p$ given in the remark guarantees that $G_p$ is convex-cocompact, invoking \Cref{prop:cc_edge->vertex} and using that ${\rm Ess}(\mc{Z}(G),\emptyset)\sq\mc{Z}(G)$. (We need this roundabout argument because the elements of ${\rm Per}(\mc{Z}(G),\mscr{E})$ are not all convex-cocompact in general.) 

    Regarding part~(2) of the corollary, \Cref{thm:acc_implies_nice}(2) shows that there are only finitely many conjugacy classes of point-stabilisers of $T$ that are not arc-stabilisers, and so it suffices to show that arc-stabilisers of $T$ also fall into finitely many conjugacy classes. Consider an arc $[x,y]\sq T$. Since chains in $\mc{Z}(G)$ have bounded length, there exist two points $x',y'\in(x,y)$ such that the arcs $[x,y]$ and $[x',y']$ have the same $G$--stabiliser (with the points $x,x',y',y$ so aligned). Choosing stable arcs $\beta_x\sq[x,x']$ and $\beta_y\sq[y',y]$, we see that the stabiliser of $[x,y]$ coincides with the intersection of the stabilisers of $\beta_x$ and $\beta_y$. Now, there are only finitely many conjugacy classes of stabilisers of stable arcs, by \Cref{thm:acc_implies_nice}(1), and so \Cref{lem:cc_basics}(3) implies that there are only finitely many conjugacy classes of (general) arc-stabilisers, as we wanted.
    
    Finally, part~(4) of the corollary follows from \Cref{thm:acc_implies_nice}(3b), provided that we show that nontrivial elements of $\mc{Z}(G)\cup{\rm Per}(\mc{Z}(G),\mscr{E})$ fall into cases (4b) and (4c) of the corollary. For this, consider first a nontrivial centraliser $Z\in\mc{Z}(G)$ and recall that $Z=Z_GZ_G(Z)$. Since $Z$ is convex-cocompact, \Cref{lem:cc_basics}(2) shows that either $N_G(Z)=Z$ or $Z\leq N_G(Z)\in\mc{VP}(G)$. In the former case, we have $Z_G(Z)\leq Z$ and hence $Z=Z_G(C)$ where $C$ is the centre of $Z$. Applying \Cref{lem:cc_basics}(2) to $C$, we conclude that either $Z\in\mc{VP}(G)$ or $Z=C$. Finally, in the latter case, either $Z\cong\Z$ or $Z\in\mc{VP}(G)$. Consider now an element $P\in{\rm Per}(\mc{Z}(G),\mscr{E})$ and let $F\in\mc{Z}(G)$ be the fibre of the corresponding QH subgroup. We have $P\leq N_G(F)$ and so, if $F\neq\{1\}$, then another application of \Cref{lem:cc_basics}(2) shows that either $N_G(F)\in\mc{VP}(G)$ or $N_G(F)=F\in\mc{Z}(G)$. Finally, if $F=\{1\}$, then $P\cong\Z$ and hence $P\leq Z_G(P)\in\mc{Z}(G)$. In conclusion, all nontrivial elements of $\mc{Z}(G)\cup{\rm Per}(\mc{Z}(G),\mscr{E})$ are either cyclic centralisers or subgroups of elements of $\mc{VP}(G)$, as required. 
    
    This concludes the proof of the corollary.
\end{proof}

\bibliography{./mybib}
\bibliographystyle{alpha}

\end{document}